\documentclass[letterpaper, 10 pt, conference]{ieeeconf} 
\IEEEoverridecommandlockouts                                                                                   
\usepackage{graphics} 
\usepackage{epsfig} 
\usepackage{amsmath}
\usepackage{amsfonts} 
\usepackage{xcolor}
\usepackage{balance}
\usepackage{dblfloatfix}
\usepackage{soul}
\usepackage{latexsym,theorem}
\usepackage{url}
\usepackage{cite}
{\theorembodyfont{\slshape}\newtheorem{theorem}{Theorem}[section]}
{\theorembodyfont{\slshape}}
{\theorembodyfont{\slshape}\newtheorem{lemma}[theorem]{Lemma}}
{\theorembodyfont{\slshape}}
{\theorembodyfont{\slshape}}
{\theorembodyfont{\upshape}\newtheorem{definition}[theorem]{Definition}}
{\theorembodyfont{\upshape}\newtheorem{problem}[theorem]{Problem}}
{\theorembodyfont{\upshape}\newtheorem{remark}[theorem]{Remark}}
{\theorembodyfont{\upshape}\newtheorem{assumption}[theorem]{Assumption}}
{\theorembodyfont{\upshape}}
\title{\LARGE \bf
Dynamic Output-Feedback Controller Synthesis for Dissipativity and $H_2$ Performance from Noisy Input-State Data
}
\author{Pietro Kristović$^{1}$, Andrej Jokić$^{1}$ and Mircea Lazar$^{2}$
\thanks{$^{1}$ Faculty of Mechanical Engineering and Naval Architecture, University of Zagreb, Zagreb, Croatia, {\tt\small pietro.kristovic@fsb.unizg.hr, andrej.jokic@fsb.unizg.hr}}
\thanks{$^{2}$ Electrical Engineering Faculty, Eindhoven University of Technology, Eindhoven, The Netherlands
        {\tt\small m.lazar@tue.nl}}
\thanks{This research has been supported by the European Regional Development Fund under grant agreement PK.1.1.10.0007 (DATACROSS).}
}
\begin{document}
\maketitle
\thispagestyle{empty}
\pagestyle{empty}
\begin{abstract}
In this paper we propose dynamic output-feedback controller synthesis methods  for discrete-time linear time-invariant systems. The synthesis goal is to achieve dissipativity with respect to a given quadratic supply rate or a given $H_2$ performance level. 
It is assumed that the model of system dynamics is unknown,
expect for the disturbance term.
Instead, we have a recorded trajectory of the control input and the state, which can be corrupted by an unknown but bounded disturbance. The state data is used only for the purpose of controller synthesis, while the designed controller is output feedback controller, i.e., the full state is not used for control in real time.   
The presented synthesis method is formulated  in terms of linear matrix inequalities parametrized by a scalar variable, while in noiseless
case it reduces to linear matrix inequalities. Within the considered setting, the synthesis procedure is non-conservative.

\end{abstract}
\section{INTRODUCTION}
\label{IntroductionLabel}

In this paper, we propose dynamic output-feedback controller synthesis method for discrete-time linear time-invariant (LTI) systems.
The synthesis goal is to render the closed-loop system dissipative with respect to a given generic unstructured quadratic supply rate, or to achieve a given $H_2$ performance level on a considered channel. 
Our approach is a direct data-driven
approach in sense that the recorded data of the system trajectories is used directly in the controller synthesis, avoiding reconstruction of the system model.
More precisely, we use \textit{informativity approach} framework, 
which has gained significant attention during past years, see, e.g,
\cite{R27} for an overview.

Data-driven approaches based on input-state trajectories
have been used to derive (stabilizing, $H_2$, $H_\infty$, quadratic performance) static state-feedback controller synthesis methods for generic LTI systems in \cite{R26, c4, c6, R5564636, R24, zz1, zz66, c5}.
 In contrast, in this paper we use the input-state data to derive dynamic \emph{output-feedback} controller synthesis method.
On the other hand, in \cite{R26}, \cite{c5, a1, zz2, R23}, the data of  input-output trajectories were used in  to derive (stabilizing, $H_2$, $H_\infty$, quadratic performance) dynamic output-feedback controller synthesis methods.
However, all these methods are limited to auto-regressive models 
in which the disturbance affects the system in a specific way (with no associated dynamics), thus, they cannot describe an arbitrary LTI system. 
Furthermore, these results require specific \textit{restriction} regarding the relation between number of outputs available for control, lag and order of minimal realization of the system.
This \textit{restriction}  can be resolved by constructing special non-minimal realization using the data as shown in \cite{R33}. In case the signal-to-noise ratio is sufficiently large, robust
controller synthesis methods for stabilization from \cite{c6} and
\cite{R5564636} can be used. However, it
is unclear how to define and to interpret disturbance term in the model of system dynamics. Therefore, it is also unclear how to make a step from stabilization only to an input-output performance optimization. 
The mentioned \textit{restriction} is also resolved in \cite{R22}, where the system stabilization based on data corrupted by a measurement error is considered.
This is achieved by  constructing an auxiliary
system which extends system realization. However, it is not clear how to construct this auxiliary system, as stated in  \cite[Sec. VI]{R33}. 
Furthermore, in both \cite{R33} and \cite{R22} the structure of the constructed system realization is not fully exploited for the purpose of controller synthesis, and therefore the obtained results are in general conservative.

In this paper, we consider a generic LTI system within the common problem setting, as defined in \cite{c6}.
More precisely, we assume that the matrices that describe the system dynamics are unknown, except for the matrix associated with the disturbance. This is the case
in many practical settings, i.e. we often know how a disturbance (or noise) affects the system dynamics.
 Furthermore, we assume that the  matrices of the output channel for
control and performance are known (as user's choice, see, e.g., \cite{c4} for a discussion).
To compensate for a partially unknown model, we assume to have one recorded system trajectory  consisting of the control input data and the corresponding system's state trajectory. The recorded data is possibly corrupted by an unknown but bounded (generic process) disturbance. 
We emphasize that the input-state data from a data recording process is only required for the purpose of controller design, while the designed controller is a dynamic output-feedback controller rather than a static state-feedback controller. 
That is, once the controller is designed, it uses only measurements of the output available for control. 
The derivation of the presented synthesis method partly follows the convexification procedure from the model-based synthesis, see e.g. \cite{R25}.
The final result is expressed in the form of linear matrix inequalities (LMIs) parametrized by a scalar variable, thus, controller parameters can be obtained by solving these LMIs multiple times in a line-search procedure. 
In the noiseless case, the synthesis conditions reduce to
LMIs.
Obtained synthesis procedures are non-conservative, more precisely, they represent necessary and sufficient conditions for a set of systems that are consistent with the recorded data and the known disturbance bound.
This work is a continuation of the $H_\infty$ control result in \cite{RT22}, using more common problem setting.

On the more practical side, the question is: in which cases are the derived methods appropriate?
 One might imagine situations where all states can be measured in some scenarios (that is the “data recording setting”, as we refer to it in the paper), but using all state measurements in real-time for control might not be convenient or allowed.
For example, in power systems, both voltage angles and frequencies are available for measurement, but in frequency control, only frequency measurements are typically used. In general, there might be a relatively large number of states, but the user desires to use only some (relatively small number) of them for control. Measurements of all states during the operation might be too expensive (many sensors are required to be integral part of the system) and requires suitable communication links form (possibly many) sensors to the controller, which might lead to fragility (i.e., to failures in sensing and communication) of the overall control system.

The considered problem setting is common for the synthesis of a static state-feedback controller, thus, our results is generalization of these methods to a dynamic output feedback controller.
Within the considered problem setting and to the best of authors knowledge, presented results are first dynamic output-feedback controller synthesis methods for dissipativity and $H_2$ performance of a generic LTI system.
To illustrate results, we use an academic example from \cite{c4} and active suspension example from \cite{MatRef}.

\subsubsection*{Notation} 
We use $M^\dagger $ to denote the Moore-Penrose pseudo-inverse
of a real matrix $M$.
For a real square matrix $M$ we use $\text{He}\{M\}$ to denote a symmetric matrix $M^\top + M$. We use $\text{diag}(A,B)$ to denote the matrix \begin{small}$\begin{pmatrix} A & 0 \\ 0 & B  \end{pmatrix}$\end{small}. 
With matrices
$A_1,...,A_n$ which have the same number of columns, we
use $\text{col}(A_1,..., A_n)$ to denote the matrix $\begin{pmatrix} A_1^\top \cdots A_n^\top \end{pmatrix}^\top$.
In a linear matrix inequality (LMI), $\star$ represents blocks which can be inferred from symmetry.  
When replacing a set of inequalities $C \prec 0$, $A-BC^{-1}B^\top \prec 0$ with a matrix inequality of the form \begin{small}$\begin{pmatrix} A & B \\ B^\top & C \end{pmatrix}\prec 0$\end{small} (and vice versa), we will say that we have \emph{used the Schur complement with respect to the matrix $C$}. In such a way we explicitly specify which block diagonal matrix is inverted. When we say that \textit{we apply a congruence transformation on a matrix inequality $A\succ 0$ with respect to the matrix $S$}, we mean that $S$ is a full rank matrix, and that the transformed inequality has the form $S^\top A S \succ 0$. 
 With $I_m$ and $0_{p_1 \times p_2}$ we denote the identity matrix of the size $m$, and the zero matrix of size $p_1 \times p_2$, respectively.
\section{Preliminaries}
\subsection{Dissipativity}
\label{SecDissipativity}
Consider a discrete time LTI system 
\begin{equation}
\label{Sys1}
\begin{pmatrix} x(t+1) \\ z(t) \end{pmatrix} =
\begin{pmatrix} A & B \\ C & D \end{pmatrix}
\begin{pmatrix} x(t) \\  w(t) \end{pmatrix} 
\end{equation}
where $t \in \mathbb{N}_{\geq 0}$, $x(t) \in \mathbb{R}^n$, $w(t) \in \mathbb{R}^{m_w}$ and $z(t) \in \mathbb{R}^{p_z}$ represent time, the system state, input and output, respectively. We say that \eqref{Sys1} is strictly dissipative with respect to a quadratic supply function 
\begin{equation}
\label{supplyFunction}
s(w(t),z(t)):=
\begin{pmatrix}
w(t) \\ z(t)
\end{pmatrix}^\top
\begin{pmatrix}
-Q & -S \\
-S^\top & -R
\end{pmatrix}
\begin{pmatrix}
w(t) \\ z(t)
\end{pmatrix}
\end{equation}
 if there exist a symmetric matrix $P$ such that 
\begin{equation}
\label{Dissip2}
\begin{pmatrix}
I & 0 \\
A & B \\
0 & I\\
C & D  \\
\end{pmatrix}^\top
\begin{pmatrix}
-P & 0 & 0 & 0  \\
0& P & 0 & 0   \\
0 & 0 & Q & S \\
0 & 0 & S^\top & R\\
\end{pmatrix}
\begin{pmatrix}
I & 0 \\
A & B \\
0 & I\\
C & D  \\
\end{pmatrix}
\prec 0,
\end{equation}
in which case $V(x(t))=x(t)^\top P x(t)$ acts as a storage function.
If in addition we have $P\succ 0$ and $R\succeq 0$, then \eqref{Dissip2} implies stability of system \eqref{Sys1} since the Lyapunov inequality $A^\top P A-P\prec 0$ is incorporated in \eqref{Dissip2}.
\begin{remark}
\label{RemarkPrvi1}
The channel $w  \rightarrow z $ achieves $H_\infty$ performance of at least $\gamma $ (with $\gamma >0$) if and only if \eqref{Dissip2} holds for
$Q=-\gamma^2 I_{m_w}$, $R = I_{p_z}$,  $S = 0$ and some $P \succ 0$.
\end{remark}
\subsection{$H_2$ performance}
\label{H2performanceSubsection}
Consider a discrete time LTI system \eqref{Sys1}
and a $H_2$ norm on the performance channel $w \rightarrow z$ denoted by $||T(e^{j\omega})||_2$, where
\begin{equation}
\label{PrijenosnaFunkcija}
T(e^{j\omega} )=C (e^{j\omega}I-A)^{-1}B+D.
\end{equation}
According to \cite{zz00}, $||T(e^{j\omega})||_{2} < \mu $ if and only if  $\text{tr}(Z)< \mu^2$ and
\begin{equation}
\label{H2LMI}
\begin{pmatrix}
P-A P A^\top & B \\
B^\top &  I
\end{pmatrix}
 \succ 0, 
 \quad 
 \begin{pmatrix}
Z-DD^\top  & CP \\
P  C^\top &  P
\end{pmatrix} \succ 0,
\end{equation}
in which case the channel $w \rightarrow z $ achives a $H_2$ performance of at least $\mu$. Furthermore,
\eqref{H2LMI} implies stability of system \eqref{Sys1} since Lyapunov inequalities are incorporated in \eqref{H2LMI}.
\subsection{Matrix S-lemma}
The following version of the matrix S-lemma has been presented in \cite[Thm. 4.10]{c21}. 
\begin{theorem}
\label{Theorem1}
Let $M, H \in \mathbb{R}^{(n+r)\times (n+r)}$ be symmetric matrices with partitions \begin{small}$H:=\begin{pmatrix} H_{11} & H_{12} \\ H_{12}^\top & H_{22} \end{pmatrix}$\end{small} where $H_{11}\in \mathbb{R}^{n \times n}$,
and let the set $S_{H}$ be defined as
\begin{equation}
S_{H}:=\{ Z\in\mathbb{R}^{r\times n} | \begin{pmatrix} I \\ Z \end{pmatrix}^\top H \begin{pmatrix} I \\ Z \end{pmatrix} \succeq 0 \}.
\end{equation}
Assume that $H_{22}\prec 0$ and $S_{H} \neq \emptyset$.
Then, we have that
\begin{equation}
\begin{pmatrix}
I \\
Z
\end{pmatrix}^\top
M
\begin{pmatrix}
I \\
Z
\end{pmatrix}\succ 0 \text{ for all } Z\in S_{H}
\end{equation}
if and only if there exists scalar $\alpha \geq 0$ such that
\begin{equation}
M-\alpha
H \succ
0. 
\end{equation}
\end{theorem}
Note that in Theorem~4.10 it is in fact required that $H_{11}-H_{12}H_{22}^\dagger H_{12}^\top \succeq 0$ (i.e., the generalized Schur complement of $H$ with respect to $H_{22}$ is positive semidefinite). This condition is equivalent to the condition that $S_{H}$ is not an empty set, as shown in \cite[page 6]{c21}.

\section{Problem definition}
\subsection{The plant}
Consider the following discrete time LTI system 
\begin{equation}
\label{eq:3:11}
\begin{pmatrix}
x(t+1) \\
z(t)\\
y(t)
\end{pmatrix}
=
\begin{pmatrix}
A & B_1 & B \\
C_1 & D_1 & E  \\
C & F & 0 
\end{pmatrix}
\begin{pmatrix}
x(t)\\
w(t)\\
u(t) \\
\end{pmatrix},
\end{equation}
where $x(t)\in \mathbb{R}^n$ is the state vector, $w(t) \in \mathbb{R}^{m_w}$ is the disturbance input, $u(t) \in \mathbb{R}^m$ is the control input, $z(t) \in \mathbb{R}^{p_z}$ is the controlled output, $y(t) \in \mathbb{R}^p$ is the measurement signal which is available for control.
Next, consider the following dynamic output-feedback controller
\begin{equation}
\label{controller}
\begin{pmatrix}
x_{c} (t+1)\\
u (t)
\end{pmatrix}
=
\begin{pmatrix}
A_{c} & B_c \\
C_c & D_c
\end{pmatrix}
\begin{pmatrix}
x_{c} (t) \\
y (t)
\end{pmatrix},
\end{equation}
where the controller is of the same order as the plant, i.e.  $x_{c} (t) \in \mathbb{R}^n$. 
The system \eqref{eq:3:11} and the controller \eqref{controller} 
form the closed-loop system
\begin{equation}
\label{closed_loop}
\begin{pmatrix}
\xi (t+1) \\
z (t)
\end{pmatrix}
=
\begin{pmatrix}
\tilde{A} & \tilde{B} \\
\tilde{C} & \tilde{D} 
\end{pmatrix}
\begin{pmatrix}
\xi (t) \\ w (t)
\end{pmatrix}
\end{equation}
where $\xi (t)=\text{col}(x (t),x_{c} (t))$  and
\begin{equation}
\label{closed_loop_matrices}
\begin{pmatrix}
\tilde{A} & \tilde{B} \\
\tilde{C} & \tilde{D} 
\end{pmatrix}
=
\begin{pmatrix}
\begin{array}{cc|c}
A+BD_cC & BC_c & B_1+BD_cF \\
B_cC & A_c & B_cF \\ \hline
C_1+ED_cC & EC_c & D_1+ED_cF 
\end{array}
\end{pmatrix}.
\end{equation}
\subsection{Problem and data recording settings}
In our problem setting we assume that the matrices:
\begin{enumerate}
\item $A$ and $B$ are unknown,
\item $B_1$, $C_1$, $D_1$, $E$, $C$ and $F$ are known, and $B_1$ has full column rank.
\end{enumerate}
In addition, we assume that we have one recorded  finite time trajectory.
Note that trajectory data can be also made of multiple recorded trajectories or their parts without affecting controller synthesis.
This data is obtained in the following data recording setting:
\begin{enumerate}
\item the system model structure \eqref{eq:3:11} is known; 
\item the bound of the unknown disturbance $w$ is known;
\item the control input $u$ and the state $x$ are known.
\end{enumerate}  
More precisely, we have knowledge of the following data matrices
\begin{equation}
\label{RecordedData}
\begin{split}
X_{+}&:=
\begin{pmatrix}
x (1) & x (2) & \cdots & x (N)
\end{pmatrix}, \\
X_- &:=
\begin{pmatrix}
x (0) & x (1) & \cdots & x(N-1)
\end{pmatrix}, \\
U_-&:=
\begin{pmatrix}
u (0) & u (1) & \cdots & u(N-1)
\end{pmatrix}, 
\end{split}
\end{equation} 
which satisfy the equation
\begin{equation}
\label{eq:2:8}
X_{+}
= A X_-+BU_-+B_1W_-, 
\end{equation}
where the matrix $W_-$ represents an unknown disturbance data   
\begin{equation}
\label{NoiseData}
W_{-}
:=
\begin{pmatrix}
w (0) & w (1) & \cdots & w(N-1)
\end{pmatrix}.
\end{equation}

\subsection{The disturbance model}
\label{The disturbance modelSubsection}
The disturbance affecting data recording process is assumed to be bounded. This is modeled using a quadratic matrix inequality (QMI) imposed on $W_-$. 
\begin{assumption}
\label{Assumption1}
The matrix $W$ from \eqref{NoiseData} satisfies the inequality  
\begin{equation}
\label{eq:2:6}
\begin{pmatrix}
I \\
W^\top\\
\end{pmatrix}^\top
\underbrace{
\begin{pmatrix}
\Phi_{11} & \Phi_{12}\\
\Phi_{12}^\top & \Phi_{22}
\end{pmatrix}}_\Phi
\begin{pmatrix}
I \\
W^\top\\
\end{pmatrix} \succeq 0,
\end{equation}
where $\Phi=\Phi^\top \in \mathbb{R}^{(m_w+N)\times(m_w+N)}$ is a known matrix with $\Phi_{22} \prec 0$.~$\hfill{\Box}$   
\end{assumption}
Note that $\Phi_{22} \prec 0$ ensures that the set of matrices $W_-$ which satisfy \eqref{eq:2:6} is bounded \cite{c6}.

\subsection{The set of plants consistent with the data}
Let $\Sigma$ be the set of all pairs $(A, B)$ which can explain the data
($X_+$,$X_-$,$U_-$), that is
\begin{equation}
\label{SystemSet}
\begin{split}
\Sigma :=
\{(A,B) \,| 
 \eqref{eq:2:8} \text{ holds for some $W_-$ satisfying } \eqref{eq:2:6} \}. 
\end{split}
\end{equation}
The following lemma is used to express the set $\Sigma$ in a (QMI) form  suitable for the application of the matrix S-lemma. 
\begin{lemma}
\label{Lemma3}
The set $\Sigma$ is the set of all pairs ($A$,$B$) that satisfy the following inequality
\begin{equation}
\label{eq:2:3}
\begin{pmatrix}
I   \\ 
A^\top  \\
B^\top   
\end{pmatrix}^\top
\underbrace{
\begin{pmatrix}
H_{11}  & H_{12}  & H_{12} \\
H_{12}^\top & H_{22}  & H_{23} \\
H_{12}^\top & H_{23}^\top & H_{33}
 \end{pmatrix}}_{H}
\begin{pmatrix}
I   \\ 
A^\top  \\
B^\top 
\end{pmatrix} \succeq 0,
\end{equation}
where
\begin{equation}
\label{PhiDef}
H
:=
\begin{pmatrix}
\star
\end{pmatrix}
\begin{pmatrix}
\Phi_{11} &  \Phi_{12}\\ 
\Phi_{12}^\top  & \Phi_{22}
\end{pmatrix}
\begin{pmatrix}
\begin{array}{c|c|c}
 B_1^\top & 0 & 0\\
X_+^\top & -X_-^\top & -U_-^\top
\end{array}
\end{pmatrix}.
\end{equation}~$\hfill{\Box}$  
\end{lemma}
\begin{proof}
If we multiply \eqref{eq:2:6} with matrices $B_1$ and $B_1^\top$ from the left and right side, respectively, then we obtain an equivalent matrix inequality since matrix $B_1$ has a full column rank. Next, if we substitute  $B_1W_-$  using \eqref{eq:2:8} we obtain \eqref{eq:2:3}.
\end{proof}

The following statement ensures that all assumptions regarding the matrix S-lemma are met.
\begin{assumption}
\label{Assumption2} The set $\Sigma$ is a nonempty set and we have that
$\text{rank}(\text{col}(X_-,U_-))=n+m.$~$\hfill{\Box}$   
\end{assumption}
According to \cite{c1}, this rank condition is achieved when the system is controllable  and
persistently exciting inputs are used during the data recording process.
More precisely, this is true if pair $(A, B)$ is controllable and input $u$ is 
persistently exciting of order $n+1$.
Note that Assumptions~\ref{Assumption2} implies that the set $\Sigma$ is bounded.

The following definition is used to construct controller synthesis procedures.
\begin{definition}
\label{Definicija1}
Consider \eqref{eq:2:3} and let 
$\Lambda \in \mathbb{R}^{\tilde{n} \times \tilde{n}}$ be a positive definite matrix defined by arbitrary factorization
\begin{equation}
 \nonumber
\begin{pmatrix}
\star \\
\end{pmatrix}
H
\begin{pmatrix}
I & A_s & B_s \\
\end{pmatrix}^\top=\Xi \Lambda \Xi^\top,
\end{equation}
where
\begin{equation}
\label{set_element}
\begin{pmatrix}
A_s & B_s \\
 \end{pmatrix}^\top
 :=
 -\begin{pmatrix}
  H_{22} &  H_{23} \\
 H_{23}^\top & H_{33}
\end{pmatrix}^{-1}
\begin{pmatrix}
H_{12}^\top \\
H_{13}^\top 
\end{pmatrix}.
\end{equation}~$\hfill{\Box}$
\end{definition} 
According to the argumentation in \cite[expr. (3.4)]{c21},
 equation \eqref{set_element}  can be used to compute a
pair $(A_s,B_s) \in \Sigma$.

\subsection{The control problem}
\label{TheControlProblem}
In this paper we solve the following control problem.
\begin{problem}
\label{Problem1}
Consider the system \eqref{eq:3:11}, and suppose that we have the following knowledge: matrices $B_1$, $C_1$, $D_1$, $E$, $C$ and $F$; the recorded data ($X_+$,$X_-$,$U_-$) defined as in \eqref{RecordedData};  the disturbance bound specified by $\Phi$ in \eqref{eq:2:6}. 
Suppose that Assumptions~\ref{Assumption1} and \ref{Assumption2} hold.
Design a dynamic output-feedback controller \eqref{controller} such that the closed-loop system \eqref{closed_loop} achieves one of the following two goals:
\begin{enumerate}
 \item[\emph{a})]
is stable and strictly dissipative
 with respect to a given supply function \eqref{supplyFunction} for all $(A, B) \in \Sigma$.  
 \item[\emph{b})] 
 achieves a given $H_2$ performance $\mu$ on a given  channel $w \rightarrow z$ for all $(A, B) \in \Sigma$ .~$\hfill{\Box}$
\end{enumerate}     
\end{problem}
Note that disturbance $w$ in practice can be 
arbitrary, 
 but during the data recording we desire to keep it small to reduce the set $\Sigma$ and therefore to possibly achieve better performance levels of the closed-loop 
 (e.g. $H_\infty$ or $H_2$ performance).     
\section{Controller synthesis}
\label{SolutionProblem1}
Here we present a solution to the Problem~\ref{Problem1}
in the form of two theorems which present the main results of this paper.
 Note that the statements in these theorems are in accordance with  dissipativity and $H_2$ performance description presented in Sections~\ref{SecDissipativity} and \ref{H2performanceSubsection}, respectively.
 
\begin{figure*}[b]
 \par\noindent\rule{\textwidth}{0.5pt}
\begin{equation}
\tag{CS}
\label{eq:3:5}
\begin{gathered}
\Pi:=
\begin{pmatrix}
\Pi_{11} & \Pi_{12} \\
\Pi_{12}^\top & \Pi_{22} 
\end{pmatrix}, \quad
\Pi_{11}:=
\begin{pmatrix}
X & I & \star  & \star   \\
I & Y & \star   & \star\\
-S(C_1+ED_cC) &  -S(C_1Y+E\tilde{C}_c) & -Q-\text{He}(S(D_1+ED_cF))  &   \star \\
T^\top(C_1+ED_cC) & T^\top(C_1Y+E\tilde{C}_c) & T^\top(D_1+ED_cF) & \tilde{R}^{-1}
\end{pmatrix}, \\
\Pi_{12}^\top:= 
\begin{pmatrix}
A_s+B_sD_cC & A_sY+B_s\tilde{C}_c  & B_1+B_sD_cF & 0_{n\times \tilde{p}_z} \\
XA_s+\tilde{B}_cC& \tilde{A}_c   &  XB_1 +\tilde{B}_cF  & 0 \\
I & Y & 0 & 0 \\
D_cC & \tilde{C}_c & D_cF & 0 \\
0_{\tilde{n} \times n} & 0 & 0 & 0
\end{pmatrix}, 
\\
\Pi_{22}:=
\begin{pmatrix}
Y & I & -\alpha (H_{12}+A_s  H_{22}+B_s H_{23}^\top) & -\alpha (H_{13}+A_s  H_{23}+B_s H_{33})  &  \alpha \Xi \\
I & X &  -\alpha X(H_{12}+A_s  H_{22}+B_s   H_{23}^\top) & -\alpha X(H_{13}+A_s  H_{23}+B_s H_{33})  &  \alpha X \Xi  \\
\star & \star &   -\alpha  H_{22} & -\alpha  H_{23} & 0 \\
\star & \star &  \star & -\alpha  H_{33} & 0  \\
\star & \star &  \star & \star &  \alpha \Lambda^{-1}  \\
\end{pmatrix}, \\
\Delta :=
\begin{pmatrix}
I & D_1^\top+F^\top D_c^\top E^\top & 0 & 0 \\
\star & Z & C_1+ED_cC & C_1 Y+E \tilde{C}_c \\
\star & \star & X & I \\
\star & \star & \star & Y
\end{pmatrix}.
\end{gathered}
\end{equation}
\end{figure*}
\begin{theorem}
\label{Theorem2}
Consider the Problem~\ref{Problem1}a, the Definition~\ref{Definicija1},
 and  the matrix  $\Pi$ defined in expression \eqref{eq:3:5} presented at the bottom of next page. Let  $R\succeq 0$ and $\tilde{R} \in \mathbb{R}^{\tilde{p}_z \times \tilde{p}_z  }$ be a positive definite matrix defined by arbitrary factorization  $R=T\tilde{R} T^\top$.
Then, the following two statements are equivalent:
\begin{enumerate}
 \item[\emph{i})]
The inequalities
 \begin{subequations}
 \label{eq:3:4}
 \begin{align}
 & P \succ 0, \label{E1} \\ 
 & \begin{pmatrix}
\star
\end{pmatrix}^\top
\begin{pmatrix}
-P & 0 & 0 & 0  \\
0& P & 0 & 0   \\
0 & 0 & Q & S\\
0 & 0 & S^\top & R\\
\end{pmatrix}\begin{pmatrix}
I & 0 \\
\tilde{A} & \tilde{B} \\
0 & I\\
\tilde{C} & \tilde{D}  \\
\end{pmatrix}
\prec 0, \label{E2}
 \end{align}
 \end{subequations}
hold for all $(A,B)\in \Sigma$.
\item[\emph{ii})]
The inequality
\begin{equation}
\label{eq:600:301}
\Pi \succ 0,
\end{equation}
holds for some real scalar $\alpha$, symmetric matrices $X$ and $Y$, unstructured matrices $\tilde{A}_c$, $\tilde{B}_c$, $\tilde{C}_c$ and $D_c$. 
\end{enumerate}
Furthermore, the controller parameters can be reconstructed from \eqref{eq:600:301} as follows
\begin{equation}
\label{eq:3:18}
\begin{gathered}
V^\top=U^{-1}(I-XY), \\
B_c=U^{-1}(\tilde{B}_c-XB_sD_c), \\
C_c=(\tilde{C}_c-D_cCY)V^{-\top}, \\
A_c=U^{-1}(\tilde{A}_c-X(A_sY+B_s\tilde{C}_c)-UB_cCY)V^{-\top},
\end{gathered}
\end{equation}
where $U$ is an arbitrary full rank matrix. 
\end{theorem}
\begin{proof}
The proof is divided in several steps. 
In step 1, we rewrite \eqref{eq:3:4} in a form suitable for applying Theorem~\ref{Theorem1}. In step 2, we apply Theorem~\ref{Theorem1}. In step 3, we transform derived conditions to obtain \eqref{eq:600:301} and \eqref{eq:3:18}, thus, prove necessity. In step 4, we provide the final argumentation for sufficiency.

\noindent \emph{Step 1.}
Suppose that  \eqref{eq:3:4} holds.
By applying the Schur complement rule on \eqref{E2} with respect to the matrix $\text{diag}(-\tilde{R},-P)$, we obtain the equivalent expression to \eqref{eq:3:4}:
\begin{equation}
\label{eq:3:12}
\Theta:=
\begin{pmatrix}
\Theta_{11} & \star\\  
\begin{pmatrix}
T^\top \tilde{C} & T^\top \tilde{D} \\
\tilde{A} & \tilde{B}
\end{pmatrix} & \begin{pmatrix}
\tilde{R}^{-1} & 0 \\
0 & P^{-1}
\end{pmatrix}
\end{pmatrix}
 \succ 0,
\end{equation}
where 
\begin{equation}
\nonumber
\Theta_{11}:=
\begin{pmatrix}
\star
\end{pmatrix}^\top
\begin{pmatrix}
P & 0 & 0 \\
0 & -Q & -S \\
0 & -S^\top & 0
\end{pmatrix}
\begin{pmatrix}
I & 0 \\
0 & I \\
\tilde{C} & \tilde{D}
\end{pmatrix}.
\end{equation}
Let 
\begin{equation}
\label{P_MatrixStructure}
\begin{gathered}
P:=
\begin{pmatrix}
X  & U \\
U^\top & M_{1}
\end{pmatrix}, \quad 
P^{-1}:=
\begin{pmatrix}
Y & V \\
V^\top & M_{2}
\end{pmatrix}
, \nonumber 
\end{gathered}
\end{equation}
where matrices $X$, $M_1$, $Y$ and $M_2$ are of size $n\times n$.
Note that according to the multiple controller realization argumentation shown in 
the Appendix, Section~\ref{AppendixA}, we can choose the matrix $U$. For the purpose of constructing controller synthesis we
assume that $U$ is an arbitrary regular matrix.

Next, we apply a congruence transformation on \eqref{eq:3:12} with respect to the matrix
$S_1:=\text{diag }(I,\begin{small} \begin{pmatrix} 0 & I_n \\ I_{n} & 0 \end{pmatrix} \end{small} )$ in order to obtain the equivalent inequality $\tilde{\Theta}:=S_1^\top\Theta S_1\succ 0$. Let $\tilde{\Theta}_{11}$ be the upper left block diagonal matrix of $\tilde{\Theta}$ with size $(3n+m_w+\tilde{p}_z) \times  (3n+m_w+\tilde{p}_z)$. Then, by applying the Schur complement rule on $\tilde{\Theta} \succ 0$ with respect to the matrix $\tilde{\Theta}_{11}$, we obtain the equivalent condition
\begin{equation}
\label{eq:300:13}
\Gamma
(\begin{pmatrix} Y & 0 \\ 0 & 0_{n+m} \end{pmatrix}-
\tilde{\Gamma} \tilde{\Theta}_{11}^{-1} 
\tilde{\Gamma}^\top
)
\Gamma^\top 
\succ 0, \quad \tilde{\Theta}_{11}\succ 0,
\end{equation}
where $\Gamma:=\begin{pmatrix} I & A & B \end{pmatrix}$ and
\begin{equation}
\nonumber
\tilde{\Gamma}:=
\begin{pmatrix}
0 & 0 & B_1 & 0_{n\times \tilde{p}_z} & V \\
I & 0 & 0 & 0 & 0  \\
D_cC & C_c & D_cF & 0 & 0 \\
\end{pmatrix}.
\end{equation}

\noindent \emph{Step 2.}
Note that according to Lemma~\ref{Lemma3}, inequality \eqref{eq:2:3} defines the set of all pairs
$(A,B) \in \Sigma$. Furthermore, the first matrix inequality in \eqref{eq:300:13}
and QMI \eqref{eq:2:3}
 are in a form suitable for the application of Theorem~\ref{Theorem1}.
Also, Assumptions~\ref{Assumption1}, and \ref{Assumption2} imply that the matrix \begin{small} $\begin{pmatrix}
H_{22} & H_{23} \\ 
H_{23}^\top & H_{33} \end{pmatrix}$ \end{small} defined in \eqref{eq:2:3} is negative definite, thus, $\Sigma$ is a nonempty and bounded  set.
Therefore, we can use the Theorem~\ref{Theorem1} to state that \eqref{eq:300:13} holds for all pairs $(A,B) $ that satisfy the QMI \eqref{eq:2:3}, if and only if there exist $\alpha \geq 0$  such that
\begin{equation}
\label{eq:3:15}
\begin{pmatrix} Y & 0 \\ 0 & 0_{n+m} \end{pmatrix}-\alpha H -
\tilde{\Gamma} \tilde{\Theta}_{11}^{-1} 
\tilde{\Gamma}^\top 
 \succ 0, \quad \tilde{\Theta}_{11} \succ 0.
\end{equation}

\noindent \emph{Step 3.}
Recall the Definition~\ref{Definicija1}. By applying a congruence transformation on the first inequality in \eqref{eq:3:15} with respect to the matrix
\begin{equation}
S_2:=
\begin{pmatrix}
 I & 0 & 0 \\
 A_s^\top & I & 0 \\
 B_s^\top & 0 & I
\end{pmatrix}, \nonumber
\end{equation}
we obtain the equivalent expression
\begin{equation}
\label{eq:333:665}
\begin{pmatrix} Y & 0 \\ 0 & 0_{n+m} \end{pmatrix}-\alpha \bar{H} -
\bar{\Gamma} \tilde{\Theta}_{11}^{-1} 
\bar{\Gamma}^\top 
 \succ 0, \quad \tilde{\Theta}_{11} \succ 0,
\end{equation}
where
\begin{equation}
\nonumber
\begin{split}
\bar{H}:&=
\begin{pmatrix}
 \Xi \Lambda \Xi^\top & \star &   \star  \\ 
  H_{12}^\top+ H_{22}A_s^\top+ H_{23}B_s^\top 
 &  H_{22} &  \star \\
  H_{13}^\top+H_{23}^\top A_s^\top +H_{33} B_s^\top &  H_{23}^\top & H_{33}
\end{pmatrix}, \\
\bar{\Gamma}:&=
\begin{pmatrix}
A_s+B_sD_cC & B_sC_c & B_1+B_sD_cF & 0_{n\times \tilde{p}_z} & V \\
I & 0 & 0 & 0 & 0 \\
D_cC & C_c & D_cF & 0 & 0
\end{pmatrix}.
\end{split}
\end{equation}
By applying the Schur complement rule on \eqref{eq:333:665} first with respect to the matrix $\tilde{\Theta}_{11}^{-1}$, and then to the matrix $\frac{1}{\alpha} \Lambda$ (note that $\eqref{eq:333:665}$ implies that $\alpha>0$), we obtain the equivalent expression
\begin{equation}
\label{eq300:600}
\begin{pmatrix}
\tilde{\Theta}_{11}  & \bar{\Gamma}^\top  & 0 \\
\bar{\Gamma} & \bar{H}_{\Lambda}  & \Gamma_\Lambda^\top \\
0 & \Gamma_\Lambda & \alpha \Lambda^{-1}
\end{pmatrix} \succ 0, \quad 
\end{equation}
where $\bar{H}_{\Lambda}:=\text{diag}(Y,0_{n+m})-\alpha \underline{H}$, $\underline{H}$ is the matrix $\bar{H}$ with a zero matrix instead of the matrix $\Xi \Lambda \Xi^\top$, and $\Gamma_\Lambda:=\begin{pmatrix} \alpha \Xi^\top & 0 \end{pmatrix}$.

Recall that $U$ is an arbitrary regular matrix and note that 
\begin{equation}
\label{eq:rekonstrukcija}
P P^{-1}=
\begin{pmatrix} 
XY+UV^\top & XV+UM_2\\
U^\top Y+M_1V^\top  & U^\top V+M_1M_2
\end{pmatrix}=
\begin{pmatrix}
I & 0 \\
0 & I
\end{pmatrix}.
\end{equation}
Thus, if we apply a congruence transformation on \eqref{eq300:600} with respect to matrix
\begin{equation}
\nonumber
S_3:=
\text{diag}( 
I_{2n+m_w+\tilde{p}_z}, 
\begin{pmatrix}
0 & U^\top \\
I & X 
\end{pmatrix},
I_{n+m+\tilde{n}}
),
\end{equation}
it follows that
\begin{equation}
\label{Statement111}
\begin{pmatrix}
0 & I\\
U & X \\
\end{pmatrix}
\begin{pmatrix}
M_2 & V^\top \\
V & Y
\end{pmatrix}
\begin{pmatrix}
0 & U^\top \\
I & X
\end{pmatrix}
=
\begin{pmatrix}
Y & I \\
I & X
\end{pmatrix} \succ 0.
\end{equation}
Next, if we use the Schur complement rule on \eqref{Statement111} with respect to the matrix $X$, we obtain the equivalent statement $Y-X^{-1}\succ 0$ and $X\succ 0$. From there it follows that $I-XY$ is a non-singular matrix.  Since \eqref{eq:rekonstrukcija} implies that $V^\top =U^{-1}(I-XY)$, we can conclude that the inequality \eqref{eq300:600} implies that the matrix $V$ has full rank.
Therefore, we can apply a congruence transformation on \eqref{eq300:600} with respect to matrix
\begin{equation}
\nonumber
S_4:=
\text{diag}( 
\begin{pmatrix}
I & Y \\
0 & V^\top
\end{pmatrix}
, I_{m_w+\tilde{p}_z}, 
\begin{pmatrix}
0 & U^\top \\
I & X 
\end{pmatrix},
I_{n+m+\tilde{n}}
),
\end{equation}
to obtain the equivalent inequality \eqref{eq:600:301}, this proves necessity. In order to obtain \eqref{eq:600:301} and \eqref{eq:3:18}, we used \eqref{eq:rekonstrukcija}, \eqref{Statement111} and the following set of equalities 
\begin{equation}
\nonumber
\begin{pmatrix}
I & 0 \\
Y & V
\end{pmatrix} 
\begin{pmatrix}
X & U  \\
U^\top & M_{1} 
\end{pmatrix}
\begin{pmatrix}
I & Y\\
0 & V^T
\end{pmatrix} 
=
\begin{pmatrix}
X & I\\
I & Y
\end{pmatrix},
\end{equation}
\begin{equation}
\nonumber
\tilde{C}
\begin{pmatrix}
I & Y \\
0 & V^\top
\end{pmatrix}=
\begin{pmatrix}
C_1+ED_cC & C_1Y+E\tilde{C}_c
\end{pmatrix}, 
\end{equation}
\begin{equation}
\nonumber
\tilde{C}_c:=D_cCY+C_cV^\top,
\end{equation}
\begin{equation}
\nonumber
\begin{pmatrix}
I & 0 \\
D_cC & C_c
\end{pmatrix}
\begin{pmatrix}
I & Y \\
0 & V^\top
\end{pmatrix}=
\begin{pmatrix}
I & Y\\
D_cC & \tilde{C}_c
\end{pmatrix},
\end{equation}
\begin{equation}
\nonumber
\begin{split}
&
\begin{pmatrix}
0 & I \\
U & X
\end{pmatrix}
\begin{pmatrix}
B_cC & A_c \\
A_s+B_sD_cC & B_sC_c
\end{pmatrix}
\begin{pmatrix}
I & Y \\
0 & V^\top
\end{pmatrix}= \\
&
\begin{pmatrix}
A_s+B_sD_cC & A_sY+B_s\tilde{C}_c \\
XA_s+\tilde{B}_cC& \tilde{A}_c  
\end{pmatrix},
\end{split}
\end{equation}
\begin{equation}
\nonumber
\tilde{B}_c:=
XB_sD_c+UB_c, 
\end{equation}
\begin{equation}
\nonumber
\tilde{A}_c :=X(A_sY+B_s\tilde{C}_c)+U(B_cCY+A_cV^\top),
\end{equation}
\begin{equation}
\nonumber
\begin{pmatrix}
0 & I \\
U & X
\end{pmatrix}
\begin{pmatrix}
 B_cF
    \\ B_1+B_sD_cF
 \end{pmatrix}= \begin{pmatrix}
B_1+B_sD_cF \\ XB_1 +\tilde{B}_cF
\end{pmatrix},
\end{equation}
\begin{equation}
\nonumber
\begin{split}
\bar{H}_{12\alpha}:&= -\alpha (H_{12}+A_s  H_{22}+B_s   H_{23}^\top),
\\
\bar{H}_{13\alpha}:& =-\alpha (H_{13}+A_s  H_{23}+B_s   H_{33}),
\end{split}
\end{equation}
\begin{equation}
\nonumber
\begin{pmatrix}
0 & I \\
U & X
\end{pmatrix}
\begin{pmatrix}
0 & 0  \\
\bar{H}_{12\alpha} & \bar{H}_{13\alpha}  
 \end{pmatrix}= 
 \begin{pmatrix}
\bar{H}_{12\alpha} & \bar{H}_{13\alpha}  \\
X\bar{H}_{12\alpha} & X\bar{H}_{13\alpha}  
 \end{pmatrix} ,
\end{equation}
\begin{equation}
\nonumber
\begin{pmatrix}
0 & I \\
U & X
\end{pmatrix}
\begin{pmatrix}
 0 \\
\alpha \Xi
 \end{pmatrix}= 
 \begin{pmatrix}
 \alpha \Xi \\
\alpha X\Xi
 \end{pmatrix}.
\end{equation}

\noindent \emph{Step 4.}
Suppose that \eqref{eq:600:301} and \eqref{eq:3:18} hold.
Therefore, we can construct matrix $V$ and controller matrices using \eqref{eq:3:18} and reverse the argumentation to prove sufficiency. This establishes the equivalence between statements \emph{i}) and \emph{ii}).
\end{proof}
\begin{theorem}
\label{Theorem2}
Consider the Problem~\ref{Problem1}b, the Definition~\ref{Definicija1}, and matrices  $\Pi$ and $\Delta $ defined in expression \eqref{eq:3:5} which is presented on the previous page. 
Let $Q=-I_{m_w}$ and\footnote{Note that $\tilde{p}_z$ appears in \eqref{eq:3:5} where it defines sizes of certain matrices. With $\tilde{p}_z=0$ we simply mean that these matrices are omitted.} $\tilde{p}_z=0$.
Then, the following two statements are equivalent:
\begin{enumerate}
 \item[\emph{i})]
The inequalities
 \begin{subequations}
 \label{eq:3:4_H2}
 \begin{align}
 \text{tr}(Z)  < \mu^2,& \label{E1_H2} \\ 
  \begin{pmatrix}
I & \tilde{B}^\top \\
\tilde{B} &  P-\tilde{A} P \tilde{A}^\top
\end{pmatrix}
  \succ 0,&  \label{E2_H2} \\
   \begin{pmatrix}
Z-\tilde{D}\tilde{D}^\top  &  \tilde{C}P \\
P  \tilde{C}^\top &  P
\end{pmatrix}  \succ 0, & \label{E3_H2} 
 \end{align}
 \end{subequations}
hold for all  $(A,B)\in \Sigma$.
\item[\emph{ii})]
The inequalities 
\begin{equation}
\label{eq:600:301_H2}
 \text{tr}(Z)  < \mu^2, \quad \Pi \succ 0 , \quad \Delta \succ 0,
\end{equation}
hold for some real scalar $\alpha$, symmetric matrices $X$ and $Y$, unstructured matrices $\tilde{A}_c$, $\tilde{B}_c$, $\tilde{C}_c$ and $D_c$.
\end{enumerate}
Furthermore, the controller parameters can be reconstructed from \eqref{eq:600:301_H2} as follows
\begin{equation}
\label{eq:3:18_H2}
\begin{gathered}
U=(I-XY)V^{-\top}, \\
B_c=U^{-1}(\tilde{B}_c-XB_sD_c), \\
C_c=(\tilde{C}_c-D_cCY)V^{-\top}, \\
A_c=U^{-1}(\tilde{A}_c-X(A_sY+B_s\tilde{C}_c)-UB_cCY)V^{-\top},
\end{gathered}
\end{equation}
where $V$ is an arbitrary full rank matrix. 
\end{theorem}
\begin{proof}
Suppose that  \eqref{eq:3:4_H2} holds.
By applying the Schur complement rule with respect to the matrix $P$ in \eqref{E2_H2} we obtain an equivalent expression
\begin{equation}
\label{h2Derivation1}
\Theta :=
\begin{pmatrix}
P^{-1} & 0 & \tilde{A}^\top\\
0 &  I & \tilde{B}^\top \\
\tilde{A} & \tilde{B} &  P
\end{pmatrix}\succ 0.
\end{equation}
Note that this matrix inequality is equal to \eqref{eq:3:12} with $Q=-I_{m_w}$, $\tilde{p}_z=0$  and 
\begin{equation}
\begin{gathered}
P:=
\begin{pmatrix}
Y & V \\
V^\top & M_{2}
\end{pmatrix}, \quad 
P^{-1}:=
\begin{pmatrix}
X  & U \\
U^\top & M_{1}
\end{pmatrix}. \nonumber \\
\end{gathered}
\end{equation}
Thus, we can follow same the same argumentation as in
the proof of Theorem~\ref{Theorem1} to obtain the matrix inequality $\Pi \succ 0$ from \eqref{eq:600:301_H2}.

Next, by applying the Schur complement rule with respect to the matrix $I_{m_w}$ in \eqref{E3_H2} we obtain an equivalent expression
\begin{equation}
\label{h2Derivation17}
   \begin{pmatrix}
   I & \tilde{D}^\top & 0 \\
\tilde{D}& Z  &  \tilde{C}P \\
0 & P  \tilde{C}^\top &  P
\end{pmatrix} \succ 0.
\end{equation}
Finally, by applying a congruence transformation on \eqref{h2Derivation17} with respect to the matrix $S_4:=\text{diag}(I,P^{-1}\begin{pmatrix} I & Y \\ 0 & V^\top \end{pmatrix})$ we obtain the equivalent expression $\Delta \succ 0$ from \eqref{eq:600:301_H2}, this finishes the proof.
\end{proof}
\begin{remark}
Matrix inequality  $\Pi \succ 0$ from 
\eqref{eq:600:301} and \eqref{eq:600:301_H2} is not  a LMI in the decision variables, because of the term $\alpha X$. However, by fixing $\alpha \in \mathbb{R}_{> 0}$ we get an LMI that can be solved multiple times in a line search procedure (as in \cite{c4} and   \cite{R25}).
Furthermore, if LMI is not feasible, then
the zero on the right hand side can be replaced with 
$-\varepsilon I$ and a line search can be
performed while minimizing the scalar $\varepsilon$ in the direction of decreasing its value,
until $ \varepsilon < 0$ is
eventually found, in which case inequality is feasible.
The matrix inequality $\Pi \succ 0$  differs from  LMI of the analogous model-based controller synthesis. 
Size of that LMI for model-based synthesis is $4n + m_w +\tilde{p}_z$ and $4n + m_w$ for dissipativity and $H_2$ control, respectively (see e.g. \cite[Ch. 4]{R25}). In contrast, size of our LMI is increased by $n + m + \tilde{n}$.
Although our synthesis
method belongs to direct-data driven approach, note that we
calculate parameters of one arbitrary system which is consistent with
the recorded data  (what is a simple calculation step - see expression \eqref{set_element} in Definition~\ref{Definicija1}) for  convexification of the
 synthesis. If
there is no disturbance during the data recording process,
we reconstruct the actual system. In that case, $\tilde{n}=0$ in
the Definition~\ref{Definicija1} and
\begin{equation}
\label{eq_je_nula}
\begin{pmatrix}
  H_{12}^\top+ H_{22}A_s^\top+ H_{23}B_s^\top 
  \\
  H_{13}^\top+H_{23}^\top A_s^\top +H_{33} B_s^\top 
\end{pmatrix}=0,
\end{equation}
since $\Phi_{22}$ is only non-zero matrix in $\Phi$, see
\eqref{eq:2:6}. Therefore, $\Pi \succ 0$  from \eqref{eq:600:301} and \eqref{eq:600:301_H2} is reduced to
LMI. Furthermore, if a pair $(A_s,B_s)\in \Sigma $ is calculated using \eqref{set_element}, then \eqref{eq_je_nula} also holds in the noisy case. However, a pair $(A_s,B_s)\in \Sigma $ can also be  calculated in such a way that  \eqref{eq_je_nula} is not true.

\end{remark}

\section{EXAMPLES}
To illustrate the theoretical results we use the following numerical examples.
\subsection{Robust stabilization, $H_\infty$ and $H_2$ control}
\label{example1}
Consider an unstable discrete-time system from the
$H_\infty$ control example  in \cite{c4}. This system can be described using a system model \eqref{eq:3:11}
with the matrices
\begin{equation}
\nonumber
\begin{split}
A& =
\begin{pmatrix}
-0.5 & 1.4 & 0.4 \\
-0.9 & 0.3 & -1.5 \\
1.1 & 1 & -0.4
\end{pmatrix},
\quad
B_1=
\begin{pmatrix}
0.1 & -0.3 \\ 
-0.1 & -0.7 \\
0.7 & -1
\end{pmatrix}, \\
B&=I_3, \quad C_1=I_3, \quad D_1=0, \quad E=0.
\end{split}
\end{equation}
We assume that the matrices $A$ and $B$ are not available for the controller synthesis and that not all states variables are available for the control, in particular, we assume $C=\begin{pmatrix} I_2 & 0_{2\times 1} \end{pmatrix}$, $F=0_{2\times 3}$.

We use this system model to generate exact and noisy input-state data.
The disturbance bound and the method for generating the noisy input-state data are the same as in \cite[Sec. VI-B]{c6}.
Input samples and initial condition are generated using a Gaussian distribution with zero mean and unit variance. Disturbance samples are generated in the same manner, but with standard deviation $ \sigma \in \{0, 0.05, 0.1, 0.2\}$ for exact case and noisy cases, respectively.
The disturbance bound have the following form
\begin{equation}
W_-W_-^\top\preceq 1.35 N \sigma^2 I, \nonumber
\end{equation}
where the matrix $W_-$ is defined as in \eqref{NoiseData} and $N=20$ represents number of data samples (as in \cite{c4}).
Assumptions~\ref{Assumption1} and \ref{Assumption2} regarding the data matrices were verified. 
For the controller synthesis we consider minimizations of $H_\infty$ performance $\gamma$ (as defined in Remark~\ref{RemarkPrvi1}) and $H_2$ performance $\mu$.
To solve the synthesis semidefinite programs we used Yalmip \cite{c22} environment in MATLAB, with Mosek as a LMI solver.

 \begin{figure}[b!]
\centering
\includegraphics{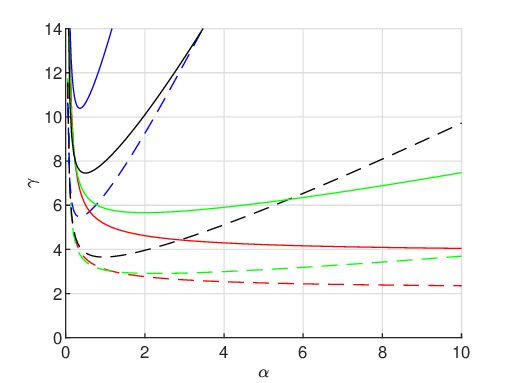}
\caption{
$H_\infty$ performance $\gamma$ as a function of the
parameter $\alpha$ using dynamic state (dashed) and output (solid) feedback controller synthesis with $\sigma \in \{0, 0.05, 0.1, 0.2\}$  in red, green, black and blue color, respectively.}
\label{fig20}
\end{figure}

The line search procedure for the $H_\infty$ controller synthesis is illustrated in the Fig.~\ref{fig20}, where $H_\infty$ performance $\gamma$ is a function of the parameter $\alpha$.
 Note that we obtain the same closed-loop performance using the model-based and the data-driven controller  with $\sigma=0$, which is, when state-feedback is considered, also consistent with the result reported in \cite{c4} (where only state-feedback was solved).  
The frequency response of the closed-loop system with obtained data-driven $H_\infty$ controller is presented in Fig.~\ref{fig22}, where by the term frequency response we refer to the largest singular value $\bar{\sigma}_{max}$ of the
corresponding transfer function $T(j\omega)$ on a particular
frequency $\omega $ (see \eqref{PrijenosnaFunkcija}). 
Additionally, numerical results for the $H_\infty$ and the $H_2$ closed-loop performance are presented in Table~\ref{Tablica1}.  
The presented results show that: (i) the developed methods achieve good robust performance in the 
case when the output is not the full state, and (ii) they recover the performance of the static state-feedback 
robust data-driven control solutions when the output coincides with the full state.

 \begin{figure}[t!]
\centering
\includegraphics{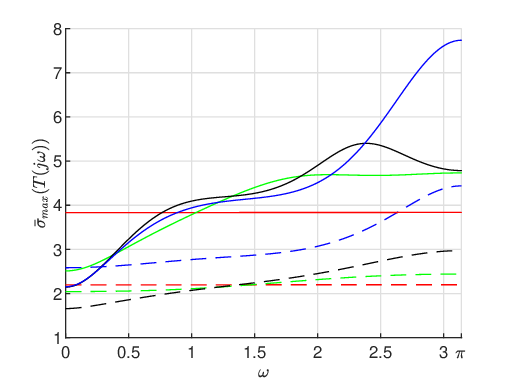}
\caption{
Frequency response of the closed-loop system using dynamic state (dashed) and output (solid) feedback controller with $\sigma \in (0, 0.05, 0.1, 0.2)$  in red, green, black and blue color, respectively.}
\label{fig22}
\end{figure}
\begin{table}[h!]
\renewcommand{\arraystretch}{1.3}
\caption{Numerical results for $H_\infty$ and $H_2$ control.}
\begin{center}
\label{Tablica1}
\begin{tabular}{ |c|c|c|c|c| } 
\hline 
 &
\multicolumn{2}{|c|}{state-feedback control} & 
\multicolumn{2}{|c|}{output-feedback control}
\\
\hline
$\sigma$  & $\gamma$ (bound) & $\mu$ (bound)  & $\gamma$ (bound) & $\mu$ (bound) \\
\hline
0  & 2.20 (2.20)   & 2.39 (2.39) & 3.84 (3.84)   & 3.78 (3.78) \\ 
0.05   & 2.44 (2.91) &   2.40 (2.49)   &  4.73 (5.66) &   3.86 (4.08)    \\ 
0.1   & 2.97 (3.66) &  2.56 (3.17)   & 5.44 (7.46) &  4.10 (4.76) 	 \\
0.2  & 4.44 (5.48)   & 3.17 (4.36)   & 7.74 (10.39)  & 4.81 (6.19) \\
\hline
\end{tabular}
\end{center}
\end{table}

\subsection{Robust $H_\infty$ control of an active suspension}
Consider a continuous-time state-space model of a quarter-car model of the active suspension system from \cite{MatRef} presented on Fig~\ref{fig40}. Using the forward Euler method we obtained the following state-space matrices
\begin{equation}
\nonumber
A=\begin{pmatrix}
1 & T_s & 0 & 0\\
-T_s\frac{k_s}{m_b} & 1-T_s\frac{b_s}{m_b} & T_s\frac{k_s}{m_b}  & T_s\frac{b_s}{m_b} \\
0 & 0 & 1 & T_s \\
\frac{k_s}{m_w} & \frac{b_s}{m_w} & \frac{-k_s-k_t}{m_w} &   1-T_s\frac{b_s}{m_w}
\end{pmatrix},
\end{equation}
\begin{equation}
\nonumber
B=\begin{pmatrix}  0\\
  T_s \frac{10^3}{m_b} \\
   0 \\
   - T_s \frac{10^3}{m_w} 
\end{pmatrix}, \quad 
B_1= \begin{pmatrix} 0 \\
 0 \\
  0  \\
   T_s k_t 
\end{pmatrix},
\end{equation}
associated with the state-space vector $\text{col}(x_b, v_b, x_w, v_w)$, the force of active component of the suspension $u \text{ [kN]}$ and  the road disturbance $w \text{ [m]}$.
Variables
$x_b \text{ [m]}$ and $v_b \text{ [m/s]}$ are the body position and velocity, respectively,
$x_w \text{ [m]}$ and $v_w \text{ [m/s]}$ are the wheel assembly position and velocity, respectively, $T_s=10^{-3}\text{s}$ is the discretization time step,  $k_s=1.6\cdot 10^4 \frac{\text{N}}{\text{m}}$ is the spring stiffness, $m_b=300 \text{kg}$  is  the quarter-car body mass, $b_s=1000 \frac{\text{N}}{\text{m/s}}$ is the damping coefficient,  $m_w=60 \text{kg}$  is  the wheel assembly mass and $k_t=1.9\cdot 10^5 \frac{\text{N}}{\text{m}}$ is the compressibility of the pneumatic tire. 
\begin{figure}[h!]
\centering
\includegraphics{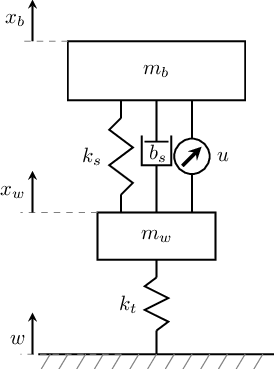}
\caption{
Active suspension system.
}
\label{fig40}
\end{figure}

The performance channel and the output available for control are defined with matrices
\begin{equation}
\nonumber
C_1=\begin{pmatrix}
1 & 0 & 0 & 0 \\
1 & 0 & -1 & 0 
\end{pmatrix}, \quad  D_1=0, \quad  E=0,
\end{equation}
and
\begin{equation}
C=\begin{pmatrix} 1 & 0 & -1 & 0 \end{pmatrix}, \quad F=0,
\end{equation}
respectively.
We consider the same problem setting as in the Section~\ref{example1}, except
the input is randomly generated from the interval  $[-10,10]$, initial conditions are equal to zero, $\sigma \in \{0,0.001\}$ for noiseless and noisy case, respectively,  and $N=2000$.

As in the previous example, derived method recovers the $H_\infty$ performance of model-based synthesis ($\gamma=1.22$) for the noiseless case ($\sigma=0$).
Furthermore, histogram on  Fig. \ref{fig45}  illustrates performance robustness using 1000 generated system trajectories for the noisy case ($\sigma=0.001$). The presented results show that the derived method achieve good robust performance results.

 \begin{figure}[t!]
\centering
\includegraphics{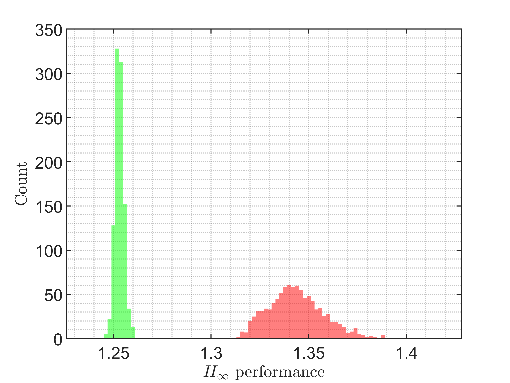}
\caption{
Green and red columns represent counts of obtained $H_\infty$ performances of an actual closed-loop systems and their bounds, respectively.}
\label{fig45}
\end{figure}

\section{CONCLUSIONS}
In this paper we have proposed non-conservative data-driven dynamic output-feedback controller synthesis methods for generic discrete-time LTI systems with the closed-loop performance criteria formalized in terms of dissipativity and $H_2$ performance.
 These are generalizations of existing static state-feedback controller
synthesis methods based on exact and noisy input-state data. 
The presented numerical examples illustrates the effectiveness of the proposed methods.
\section*{APPENDIX}
\subsection{Performance certificates and controller realizations}
\label{AppendixA}
Consider the controller synthesis LMIs \eqref{eq:3:4} and corresponding closed loop system defined by \eqref{closed_loop} and \eqref{closed_loop_matrices}.
We call the matrix $P$ that satisfies these LMIs \emph{the performance certificate}. 
Here we present some properties of performance certificates $P$. 
 These properties are instrumental in the controller synthesis procedure.  
\begin{lemma}
\label{lemma3}
Suppose the inequalities \eqref{eq:3:4} hold for the closed-loop system \eqref{closed_loop} with some controller state space realization matrices $(\bar{A}_c, \bar{B}_c, \bar{C}_c, \bar{D}_c)$ and the corresponding performance certificate \begin{small}$\bar{P}:=\begin{pmatrix} X & \bar{U} \\ \bar{U}^\top & \bar{M}_1 \end{pmatrix}$\end{small} where $X \in \mathbb{R}^{n \times n}$. Let $U \in \mathbb{R}^{n \times n}$ be an arbitrary given matrix. Then there exists a regular matrix $L\in \mathbb{R}^{n \times n} $ such that a controller space state realization matrices 
$(A_c, B_c, C_c, D_c):=(L\bar{A}_c L^{-1}, L \bar{B}_c, \bar{C}_c L^{-1}, \bar{D}_c)$ and the performance certificate \begin{small}$P:=\begin{pmatrix} X & U \\ U^\top & M_1 \end{pmatrix}$\end{small} 
satisfy \eqref{eq:3:4}.  
\end{lemma}
\begin{proof} 
Recall the matrix partition in \eqref{closed_loop_matrices} and note that we can always perturb $\bar{U}$ to obtain a regular matrix since \eqref{eq:3:4} is a strict inequality. 
Next, we apply the congruence transformation on \eqref{E1} and \eqref{E2} with respect to matrices $\text{diag}(I_n, L^{-1})$ and $\text{diag}(I_n,L^{-1}, I_{m_w} )$, respectively.
Then by using the substitutions
$(\bar{A}_c, \bar{B}_c, \bar{C}_c, \bar{D}_c)=(L^{-1}A_c L, L^{-1} B_c, C_c L, D_c)$, $U=\bar{U}L^{-1}$ and $M_1=L^{-\top}\bar{M}_1 L^{-1}$ we obtain a new set of LMIs in the form of \eqref{eq:3:4}
with the controller space state realization matrices $(A_c, B_c, C_c, D_c)$ and
the performance certificate \begin{small}$P=\begin{pmatrix} X & U \\ U^\top & M_1 \end{pmatrix}$\end{small}.
Therefore, we can conclude that by appropriate choice of $L$ we can render arbitrary $U$ from some given $\bar{U}$.      
\end{proof}
Note that the Lemma~\ref{lemma3} also holds if we consider \eqref{eq:3:4_H2} instead of  \eqref{eq:3:4}. We can prove this by following the proof of the Lemma~\ref{lemma3}
and substituting \eqref{E1}, \eqref{E2}, $\text{diag}(I_n, L^{-1})$, $\text{diag}(I_n,L^{-1}, I_{m_w} )$, $U=\bar{U}L^{-1}$, $M_1=L^{-\top}\bar{M}_1 L^{-1}$ with  \eqref{E2_H2}, \eqref{E3_H2}, $\text{diag}(I_{m_w+n}, L^\top)$, $\text{diag}(I_{p_z+n},L^{\top})$, $U=\bar{U}L^{\top}$, $M_1=L\bar{M}_1 L^{\top}$, respectively.

\balance

\bibliographystyle{IEEEtran}
\bibliography{myreferences}

\end{document}